\documentclass[preprint,12pt]{elsarticle}

\usepackage{amsmath,bm}
 \usepackage{amsthm}
 \usepackage{amsfonts,amssymb}
 \usepackage{setspace}
\usepackage{hyperref}

\numberwithin{equation}{section}

\newtheorem{theorem}{Theorem}[section]
\newtheorem{lemma}[theorem]{Lemma}
\newtheorem{remark}{Remark}
\newcommand\R{{\mathbb R}}

\journal{}

\begin{document}

\begin{frontmatter}



\title{Remark on the regularity criteria for Navier--Stokes equations in terms of one directional derivative of the velocity}

\author[a]{Hui Chen}
\ead{chenhui@zust.edu.cn}
\author[b]{Daoyuan Fang}
\ead{dyf@zju.edu.cn}
\author[b]{Ting Zhang\corref{*}}
\ead{zhangting79@zju.edu.cn}
\address[a]{School of Science, Zhejiang University of Science and Technology, Hangzhou 310023, People's Republic of China }
\address[b]{School of Mathematical Sciences, Zhejiang University,  Hangzhou 310027, People's Republic of China}
\cortext[*]{Corresponding author.}
\begin{abstract}
In this paper, we consider the 3D Navier--Stokes equations in the whole space. We investigate some new inequalities and \textit{a priori} estimates to  provide the critical regularity criteria in terms of one directional derivative of the velocity field, namely $\partial_{3} \bm{u} \in L^{p}((0,T);L^{q}(\R^3)),~\frac{2}{p}+\frac{3}{q}=2,~\frac{3}{2}<q\leq 6$.
Moreover, we extend the range of $q$ while the solution is axisymmetric, i.e. the axisymmetric solution $\bm{u}$ is regular in $(0,T]$, if $
\partial_{3} u^3 \in L^{p}((0,T);L^{q}(\R^3)),~\frac{2}{p}+\frac{3}{q}=2,~\frac{3}{2}<q< \infty.
$
\end{abstract}

\date{}
\end{frontmatter}


\section{Introduction}
In this paper, we consider the Cauchy problem of the 3D Navier--Stokes equations:
\begin{equation}
\left\{
\begin{aligned}\label{NS}
&\partial_{t}\bm{u}+(\bm{u}\cdot\nabla) \bm{u}-\Delta  \bm{u}+\nabla \Pi=0,\ (t,x)\in (0,\infty)\times\R^{3},\\
&\nabla\cdot  \bm{u}=0~,\\
&\bm{u}|_{t=0}=\bm{u}_{0}~.
\end{aligned}
\right.
\end{equation}
The solution $\bm{u}(t,x)=(u^{1},u^{2},u^{3})$, $\Pi(t,x)$ and $\bm{u}_{0}$ denote the fluid velocity field, pressure, and  the given initial data, respectively. These equations describe the flow of incompressible viscous fluid.

For given $\bm{u}_{0}\in L^{2}(\R^{3})$ with $\mathrm{div}~\bm{u}_{0}=0$ in the sense of distribution, a global weak solution $\bm{u}$ to the Navier--Stokes equations was constructed by Leray  \cite{Leray1934} and Hopf \cite{Hopf1951}, which is called Leray--Hopf weak solution. The regularity of such Leray--Hopf weak solution in three dimension plays an important role in the mathematical fluid mechanics. One essential work is usually referred as Prodi--Serrin (P--S) conditions (see \cite{Escauriaza2003,Prodi1959,Serrin1962,Takahashi1990} and the references therein.), i.e., if the weak solution $\bm{u}$  satisfies
\begin{equation}\label{PS-1}
\bm{u}\in L^{p}((0,T);L^{q}(\R^{3})),
\end{equation}
with $\frac{2}{p}+\frac{3}{q}\leq 1$, $3\leq q\leq\infty$, then the weak solution is regular in $(0,T]$.

An analogical result occurs for $\nabla \bm{u}$: the Leray--Hopf  solution $\bm{u}$ is regular, if
\begin{equation}
\nabla \bm{u} \in L^{p}((0,T);L^{q}(\R^3)),~\frac{2}{p}+\frac{3}{q}=2,~\frac{3}{2}\leq q \leq +\infty.
\end{equation}
Since then, many significant regularity criteria (See \cite{Cao2011,Chemin2016,Chemin2017,Fang2013,Han2019,Pineau2020,Qian2016} and the references therein) were established in terms of only partial components of the velocity field, or partial components of gradient of velocity field of the 3D Navier--Stokes equations. For instance, J. Y. Chemin and P. Zhang \cite{Chemin2016,Chemin2017} and B. Han et al. \cite{Han2019} proved the regularity of $\bm{u}$ in $(0,T)$, if
\begin{equation}
\int^T_0\|
u^{3}\|^p_{\dot{H}^{\frac12+\frac2p}}dt<\infty,\ \ \ p\geq 2.
\end{equation}

In this paper, we focus on the regularity criteria in terms of one directional derivative of the whole velocity field.  I. Kukavica and M. Zaine \cite{Kukavica2007} investigated the regularity criteria for the term $\partial_{3}\bm{u}$, which is scaling invariant,

\begin{equation}
\partial_{3} \bm{u} \in L^{p}((0,T);L^{q}(\R^3)),~\frac{2}{p}+\frac{3}{q}=2,~\frac{9}{4} \leq q\leq 3.
\end{equation}
Later on, C. Cao et al.\cite{Cao2010,Zhang2018,Namlyeyeva2020} extended the range of q to $q\in [1.5620,3]$. 

In this paper, we develop \textit{a priori} estimates in \cite{Cao2010,Kukavica2007}, and extend the range of q to $q\in (1.5,6]$, which is optimal on the left side. Now, we state our main theorem.
\begin{theorem}\label{thm}
	Let $\bm{u}$ be the unique solution of the Navier--Stokes equations \eqref{NS} with initial data $\bm{u}_{0}\in H^{1}(\R^{3})$ and $\mathrm{div}~\bm{u}_{0}=0$. The solution $\bm{u}$ is regular in $(0,T]$, provided that
	\begin{equation}\label{0}
	\partial_{3} \bm{u} \in L^{p}((0,T);L^{q}(\R^3)),~\frac{2}{p}+\frac{3}{q}=2,~\frac{3}{2}<q\leq 6.
	\end{equation}
	In addition, the initial data $\bm{u}_{0}$ is axisymmetric , then the solution $\bm{u}$ is regular in $(0,T]$, provided that
	\begin{equation}\label{0.1}
	\partial_{3} u^3 \in L^{p}((0,T);L^{q}(\R^3)),~\frac{2}{p}+\frac{3}{q}=2,~\frac{3}{2}<q< \infty.
	\end{equation}
\end{theorem}

\begin{remark}
	In the proof of Theorem \ref{thm}, we also prove that the solution $\bm{u}$ is regular in $(0,T]$, provided that
	$\left\|\partial_{3}\bm{u}\right\|_{L^{\infty}((0,T);L^{\frac{3}{2}}(\R^3))}\ll 1 $ is sufficient small.
\end{remark}

\begin{remark}
	Theorem \ref{thm} implies that if the Leray-Hopf weak solution $\bm{u}$  satisfies
	(\ref{0}), then the weak solution is regular in $(0,T]$. Indeed, the Leray-Hopf weak solution $\bm{u}\in L^\infty([0,T];L^2)\cap L^2(0,T;\dot{H}^2)$, then for all $s\in(0,T)$, there exists $t_0\in[0,s]$ such that $\bm{u}(t_0,\cdot)\in H^1$. From Theorem \ref{thm}, one can get the solution $\bm{u}$ is regular in $[t_0,T]$. Due to the arbitrary of $s$, one obtain that the weak solution is regular in $(0,T]$.
\end{remark}

In recent years, it has been realized that the regularity problem for axisymmetric Navier--Stokes equations is essentially a critical one under the standard scaling. Further investigations are well motivated, we refer to \cite{Chen2017,Chen2019,Kubica2012,Zhang2014} for more details.

For the regularity criteria \eqref{0}, we give a brief overview of the proof and explain some main steps.

Step 1. Anisotropic decomposition of the velocity.

Inspired by \cite{Chemin2016,Chemin2017,Han2019}, we adopt a different type of decomposition of the velocity field in \eqref{anisotropic}.
\begin{equation}
u^{h}=\Delta^{-1}\left(-\nabla_{h}\partial_{3}u^{3}+\partial^{2}_{3}u^{h}+\nabla_{h}^{\perp}\omega^{3}\right).
\end{equation}
The two-dimensional vorticity $\omega^3$ and $\partial_{3}\bm{u}$ are regarded as governing unknowns.

Step 2. Estimates of $\omega^3$ and $\partial_{3}\bm{u}$.

If we compute the time derivative of $\|\omega^3\|_{2}^2+\|\partial_{3}\bm{u}\|_{2}^2$, we need to treat the nonlinear terms ( See \eqref{1} and \eqref{K_{2}}) such
as
\begin{equation*}
\int-\partial_{3} u^{2}\partial_{1} u^{3} \omega ^{3}~\mathrm{d}x+\ldots
=\int  \partial_{3} u^{2} u^{3} \partial_{1}\omega ^{3}~\mathrm{d}x+\ldots,
\end{equation*}
since $\nabla u^3$ are bad terms in our analysis. Therefore, we have to deal with \textit{a priori} estimates involving the term $u^3$.

Step 3. Estimates of $u^3$.

This is the main part.  We work with the norm $\|\left(u^3\right)^2\|_{2}^2$ for $\frac{3}{2}<q<2$ and $\|\left(u^3\right)^\frac{3}{2}\|_{3}^2$ for $2 \leq q\leq 6$, which has the same scaling as $\|\omega^3\|_{2}^2$ and $\|\partial_{3}\bm{u}\|_{2}^2 $, respectively. To get a glimpse into this,  we assume $q=\frac{3}{2}$ in \eqref{0}, and $\left\|\partial_{3}\bm{u}\right\|_{L^{\infty}((0,T);L^{\frac{3}{2}}(\R^3))}\ll 1 $ is sufficient small.
When computing the time evolution of $\|\left(u^3\right)^2\|_{2}^2$, we need to estimate the term (See Section \ref{section 3.2} for more details)
\begin{align*}
J_{2}=&2\ \underset{\substack{i=1,2,3\\h=1,2} }{\sum}\ \int \Delta^{-1}\partial_{i}\partial_{h}\left(\partial_{3}u^{i}u^{h}\right)(u^3)^3~\mathrm{d}x.
\end{align*}
We adopt a new inequality \eqref{ingredient} to  control the term $u^h$ by $\|\Delta u^h\|_{2}$. Therefore, it is quite natural to bound the above as
\begin{align*}
J_{2}\leq& C \left\|\partial_{3}\bm{u}\right\|_{\frac{144}{85}}\left\|u^h\right\|_{18}\left\|u^3\right\|_{\frac{144}{17}}^{3}\notag\\
\leq&C\left\|\partial_{3}\bm{u}\right\|_{\frac{144}{85}}^{\frac{8}{5}}\left\|\partial_{3}u^h\right\|_{\frac{3}{2}}^{\frac{4}{9}}\left\|\Delta u^h\right\|_{2}^{\frac{5}{9}}\left\|\nabla_{h}(u^3)^2\right\|_{2}^{\frac{6}{5}}\notag\\
\leq&C \left\|\partial_{3}\bm{u}\right\|_{\frac{3}{2}}^{\frac{9}{5}}\left(\left\|\nabla\omega^3\right\|_{2}^2+\left\|\nabla(u^3)^2\right\|_{2}^2+\left\|\nabla\partial_{3}\bm{u}\right\|_{2}^{2}\right).
\end{align*}
Such estimates turn out to be crucial, and we obtain the uniform control of $E_{1}(t)=\left\|\omega^{3}\right\|_{2}^{2}+\left\|(u^{3})^2\right\|_{2}^{2}+\left\|\partial_{3}\bm{u}\right\|_{2}^{2}$
on the time interval $(0,T)$.

The rest of this paper is organized as follows. In Section 2, we set up some notations and collect a few useful lemmas. In Section 3, we obtain \textit{a priori} estimates of $\omega^3, u^3$ and $\partial_{3}\bm{u}$ for $\frac{3}{2}<q<2$. In Section 4, we obtain \textit{a priori} estimates for $2 \leq q\leq 6$. The final section is devoted to the proof of the main theorem.

\section{Notations and preliminary}
Given two comparable quantities, the inequality $X\lesssim Y$ stands for $X\leq C Y$ for some positive constant $C$. The dependence of the constant $C$ on other parameters or constants are usually clear from the context, and we will often suppress this dependence. Moreover, we denote $L^{p,q}_{T_{1},T_{2}}=L^p((T_{1},T_2);L^q(\R^3))$ and $\|\cdot\|_{r}=\|\cdot\|_{L^{r}(\R^3)}$, for the sake of simplicity.

We shall adopt the following convention for the Fourier transform:
\begin{align*}
&\hat{f}(\xi)=\int_{\mathbb{R}^{n}} f(x) e^{-i x \cdot \xi} \mathrm{d} x,\\
&f(x)=\frac{1}{(2 \pi)^{n}} \int_{\mathbb{R}^{n}} \hat{f}(\xi) e^{i x \cdot \xi} \mathrm{d} \xi.
\end{align*}
For $s\in\R$, the fractional Laplacian $\Lambda^{s}$ then corresponds to the Fourier multiplier $|\xi|^s$ defined as
\begin{equation*}
\widehat{ \Lambda^{s} f}(\xi)=|\xi|^{s} \hat{f}(\xi),
\end{equation*}
whenever it is well-defined. Analogously, we also denote anisotropic fractional Laplacian  $\Lambda_{h}^{s},\Lambda_{v}^{s}$ as
\begin{align*}
\widehat{ \Lambda_{h}^{s} f}(\xi)=|\xi_{h}|^{s} \hat{f}(\xi),\\ \widehat{ \Lambda_{v}^{s} f}(\xi)=|\xi_{3}|^{s} \hat{f}(\xi),
\end{align*}
where $\xi_{h}=(\xi_{1},\xi_{2})$ and $\xi_{3}$ are referred to the horizontal and vertical variables.

A remarkable idea introduced in J. Y. Chemin and P. Zhang \cite{Chemin2016,Chemin2017} and B. Han et al.\cite{Han2019}  is to use the decomposition of the velocity field along with horizontal and vertical directions and use the two-dimensional vorticity $\omega^3$ and $u^3$ as governing unknowns, where $\omega^3=\partial_{1}u^2-\partial_{2}u^1$.

We denote $x_h=(x_1,x_2)$,  $\nabla_{h}=(\partial_{1},\partial_{2}),~\nabla_{h}^{\perp}=(-\partial_{2},\partial_{1})$, and $u^{h}=(u^{1},u^{2})$.
To best illuminate our proof, we introduce a slightly different decomposition of the velocity field.  Notice that
\begin{equation*}
\nabla\times\nabla\times \bm{v}=\nabla\  \mathrm{div}\  \bm{v}-\Delta \bm{v},~\bm{v}=(u^{h},0).
\end{equation*}

Then, by using the Biot-Savart law, we get
\begin{equation}\label{anisotropic}
u^{h}=\Delta^{-1}\left(-\nabla_{h}\partial_{3}u^{3}+\partial^{2}_{3}u^{h}+\nabla_{h}^{\perp}\omega^{3}\right).
\end{equation}
\begin{lemma}\label{lemma2.1}
	For $1<r<\infty$, we have
	\begin{equation}
	\left\|\nabla u^{h}\right\|_{r}\lesssim\left\|\partial_{3}\bm{u}\right\|_{r}+\left\|\omega^{3}\right\|_{r},
	\end{equation}
	\begin{equation}
	\left\|\nabla^{2} u^{h}\right\|_{r}\lesssim\left\|\nabla\partial_{3}\bm{u}\right\|_{r}+\left\|\nabla\omega^{3}\right\|_{r}.
	\end{equation}
\end{lemma}
A key ingredient is introduced below.
\begin{lemma}\label{lemma2.2}
	For $f\in H^{2}(\R^3)$, we have
	\begin{equation}\label{ingredient}
	\left\|f\right\|_{b} \lesssim \left\|\partial_{3}f\right\|_{a}^{s}\cdot\left\|\nabla^2 f\right\|_{2}^{1-s},
	\end{equation}
	with $\frac{1}{2}-\frac{1}{b}=s$, $3(\frac{1}{a}-\frac{1}{2})=\frac{2}{s}-4$, $2<b<\infty,1\leq a<2$.
\end{lemma}
\begin{proof}
	Without loss of generality, we assume $f\in C_{0}^{\infty}(\R^3)$.
	
	Set $\gamma=2-4s$. By Sobolev embedding and H\"{o}lder inequality, we obtain
	\begin{align*}
	\left\|f\right\|_{b}&\lesssim \left\|\Lambda_{v}^{s}\Lambda_{h}^{2s} f\right\|_{2}\\
	&\lesssim \left\| |\xi_{3}|^{s} |\xi_{h}|^{2s}\ \hat{f}\right\|_{2}\\
	&\lesssim \left\| |\xi|^{-\gamma} |\xi_{3}|^{s}\cdot   |\xi|^{\gamma} |\xi_{h}|^{2s}\ \hat{f}\right\|_{2} \\
	&\lesssim \left\| |\xi|^{-\frac{\gamma}{s}}\cdot \xi_{3}\ \hat{f}\right\|_{2}^{s}\  \left\| |\xi|^{2}\ \hat{f}\right\|_{2}^{1-s}\\
	&\lesssim \left\|\partial_{3}f\right\|_{a}^{s} \left\|\nabla^{2} f\right\|_{2}^{1-s}.
	\end{align*}
\end{proof}
We recall the following three-dimensional Sobolev--Ladyzhenskaya inequalities (see e.g. \cite{Cao2010,Zhang2018}).
\begin{lemma}\label{lemma2.4}\label{lemmma2.3}
	For $1\leq q <\infty$, there exists a constant $C$ such that for $f\in C_{0}^{\infty}(\R^3)$,
	\begin{equation}
	\left\|f\right\|_{3q} \leq C\left\|\partial_{3} f\right\|_{q}^{\frac{1}{3}} \left\|\nabla_{h}f\right\|_{2}^{\frac{2}{3}},
	\end{equation}
	\begin{equation}
	\left\|f\right\|_{5q} \leq C\left\|\partial_{3} f\right\|_{q}^{\frac{1}{5}} \left\|\nabla_{h}\left(|f|^{2}\right)\right\|_{2}^{\frac{2}{5}}.
	\end{equation}
\end{lemma}

\section{Some \textit{a priori} estimates for $\frac{3}{2}<q<2$}
\subsection{Estimate of $\omega^3$}
Recall that $\omega^3$ satisfies the equation
\begin{equation}\label{omega^3}
\partial_{t} \omega^3+(\bm{u} \cdot \nabla) \omega^3-\Delta \omega^3= - \partial_{3} u^{2}\partial_{1} u^{3}+\partial_{3} u^{1}\partial_{2} u^{3} +\omega ^{3}\partial_{3} u^{3}.
\end{equation}
Taking $L^2$ inner product of equation \eqref{omega^3}  with $\omega^3$ , one has
\begin{equation}\label{1}
\begin{split}
\frac{1}{2}\frac{\mathrm{d}}{\mathrm{d}t}\left\|\omega^3\right\|_{2}^{2}+\left\|\nabla \omega^3\right\|_{2}^{2}&=\int-\partial_{3} u^{2}\partial_{1} u^{3} \omega ^{3}+\partial_{3} u^{1}\partial_{2} u^{3} \omega ^{3}+ \omega^{3} \partial_{3} u^{3} \omega ^{3}~\mathrm{d}x \\
&=\int  \partial_{3} u^{2} u^{3} \partial_{1}\omega ^{3}-\partial_{3} u^{1} u^{3} \partial_{2}\omega ^{3}+\frac{1}{2}\partial_{3}u^{3} \left(\omega^{3}\right)^2~\mathrm{d}x \\
&=:\ I_{1}+I_{2}+I_{3}.
\end{split}
\end{equation}
According to H\"{o}lder, interpolation and Cauchy--Schwarz inequalities,
we have
\begin{equation}\label{I_{1,2}}
\begin{split}
I_{1}+I_{2}&\leq 2 \left\|\partial_{3}u^{h}\right\|_{\frac{12q}{q+6}} \left\|u^{3}\right\|_{\frac{12q}{5q-6}}\left\|\nabla \omega^{3}\right\|_{2} \\
&\leq C \left\|\partial_{3}u^{h}\right\|_{q}^{\frac{1}{2}}\left\|\nabla\partial_{3}u^{h}\right\|_{2}^{\frac{1}{2}}\left\|(u^{3})^2\right\|_{2}^{1-\frac{3}{2q}}\left\|\nabla(u^{3})^2\right\|_{2}^{\frac{3}{2q}-\frac{1}{2}}\left\|\nabla \omega^{3}\right\|_{2} \\
&\leq C\left\|\partial_{3}\bm{u}\right\|_{q}^{p}\left\|(u^{3})^2\right\|_{2}^{2}+\frac{1}{16}\left(\left\|\nabla \omega^{3}\right\|_{2}^{2}+\left\|\nabla(u^{3})^2\right\|_{2}^{2}+\left\|\nabla\partial_{3}\bm{u}\right\|_{2}^{2}\right),
\end{split}
\end{equation}
and
\begin{equation}\label{I_{3}}
\begin{split}
I_{3}&\leq \frac{1}{2}\left\|\partial_{3} u^{3}\right\|_{q}\left\|\omega^{3}\right\|_{\frac{2q}{q-1}}^{2}\\
&\leq C \left\|\partial_{3} u^{3}\right\|_{q}\left\|\omega^{3}\right\|_{2}^{2-\frac{3}{q}}\left\|\nabla\omega^{3}\right\|_{2}^{\frac{3}{q}} \\
&\leq C \left\|\partial_{3} \bm{u}\right\|_{q}^{p}\left\|\omega^{3}\right\|_{2}^{2}+\frac{1}{16}\left\|\nabla\omega^{3}\right\|_{2}^{2}.
\end{split}
\end{equation}
Summing up \eqref{1}, \eqref{I_{1,2}} and \eqref{I_{3}}, we get
\begin{equation}\label{I}
\begin{split}
\frac{1}{2}\frac{\mathrm{d}}{\mathrm{d}t}\left\|\omega^3\right\|_{2}^{2}+\left\|\nabla \omega^3\right\|_{2}^{2}&\leq C \left\|\partial_{3} \bm{u}\right\|_{q}^{p}\left(\left\|\omega^{3}\right\|_{2}^{2}+\left\|(u^{3})^2\right\|_{2}^{2}\right)\\
&+\frac{1}{8}\left(\left\|\nabla \omega^{3}\right\|_{2}^{2}+\left\|\nabla(u^{3})^2\right\|_{2}^{2}+\left\|\nabla\partial_{3}\bm{u}\right\|_{2}^{2}\right).
\end{split}
\end{equation}

\subsection{Estimate of $u^3$} \label{section 3.2}
The equation of $u^3$ is
\begin{equation}\label{u^3}
\partial_{t}u^3+(\bm{u}\cdot\nabla) u^3-\Delta  u^3+\partial_{3} \Pi=0.
\end{equation}

Taking inner product of the equation \eqref{u^3} with $(u^3)^3$, we obtain
\begin{equation}\label{2}
\begin{split}
\frac{1}{4}\frac{\mathrm{d}}{\mathrm{d}t}\left\|(u^3)^2\right\|_{2}^{2}+\frac{3}{4}\left\|\nabla (u^3)^2\right\|_{2}^{2}=&-\int\partial_{3}\Pi\cdot (u^3)^3~\mathrm{d}x\\
=&2 \ \underset{i,j=1,2,3}{\sum}\ \int \Delta^{-1}\partial_{i}\partial_{j}\left(\partial_{3}u^{i}u^{j}\right)(u^3)^3~\mathrm{d}x \\
=&2\ \underset{i=1,2,3 }{\sum} \ \int \Delta^{-1}\partial_{i}\partial_{3}\left(\partial_{3}u^{i}u^{3}\right)(u^3)^3~\mathrm{d}x \\
&+2\ \underset{\substack{i=1,2,3\\h=1,2} }{\sum}\ \int \Delta^{-1}\partial_{i}\partial_{h}\left(\partial_{3}u^{i}u^{h}\right)(u^3)^3~\mathrm{d}x \\
=&:J_{1}+J_{2}.
\end{split}
\end{equation}
Applying H\"{o}lder, interpolation and Cauchy--Schwarz inequalities, we have
\begin{equation}\label{J_{1}}
\begin{split}
J_{1}\leq&C \left\|\partial_{3}\bm{u}\cdot u^3\right\|_{\frac{4q}{q+3}}\left\|(u^3)^3\right\|_{\frac{4q}{3(q-1)}}\\
\leq&C \left\|\partial_{3}\bm{u}\right\|_{q}\left\|(u^3)^2\right\|_{\frac{2q}{q-1}}^{2}\\
\leq& C \left\|\partial_{3} \bm{u}\right\|_{q}^{p}\left\|(u^3)^2\right\|_{2}^{2}+\frac{1}{16}\left\|\nabla(u^3)^2\right\|_{2}^{2},
\end{split}
\end{equation}
which is similar to \eqref{I_{3}}.

For $\frac{3}{2}<q<2$, we pick $s =\frac{4q}{5q+6}$, and  $\kappa=\frac{5(3-q)}{7q-3},\frac{1}{b}=\frac{1}{2}-s,\frac{1}{a}=\frac{20s+15\kappa-5}{12\kappa+20},\theta=\frac{3\kappa+15-10s}{6\kappa+10}$.
By Lemmas \ref{lemma2.1}, \ref{lemma2.2} and \ref{lemmma2.3}, we get
\begin{equation}\label{J_{2}1}
\begin{split}
J_{2}\leq& C \left\|\partial_{3}\bm{u}\right\|_{a}\left\|u^h\right\|_{b}\left\|(u^3)^2\right\|_{2}^{\frac{3-3\kappa}{2}}\left\|u^3\right\|_{5a}^{3\kappa}\\
\leq&C \left\|\partial_{3}\bm{u}\right\|_{a} \left\|\partial_{3}u^h\right\|_{q}^{s}\left\|\Delta u^{h}\right\|_{2}^{1-s}\left\|(u^3)^2\right\|_{2}^{\frac{3-3\kappa}{2}}\left\|\partial_{3}u^3\right\|_{a}^{\frac{3\kappa}{5}}\left\|\nabla_{h}(u^3)^2\right\|_{2}^{\frac{6\kappa}{5}}\\
\leq &C \left\|\partial_{3}\bm{u}\right\|_{q}^{(1+\frac{3\kappa}{5})\theta+s}\left\|\nabla\partial_{3}\bm{u}\right\|_{2}^{(1+\frac{3\kappa}{5})(1-\theta)}\left(\left\|\nabla\omega^3\right\|_{2}+\left\|\nabla\partial_{3}\bm{u}\right\|_{2}\right)^{1-s}\\
&\times\left\|(u^3)^2\right\|_{2}^{\frac{3-3\kappa}{2}}\left\|\nabla_{h}(u^3)^2\right\|_{2}^{\frac{6\kappa}{5}}\\
\leq&C\left\|\partial_{3}\bm{u}\right\|_{q}^{p}\left\|(u^3)^2\right\|_{2}^2+\frac{1}{16}\left(\left\|\nabla\omega^3\right\|_{2}^2+\left\|\nabla(u^3)^2\right\|_{2}^2+\left\|\nabla\partial_{3}\bm{u}\right\|_{2}^{2}\right).
\end{split}
\end{equation}

Summing up \eqref{2}, \eqref{J_{1}} and \eqref{J_{2}1}, we obtain
\begin{equation}\label{J}
\begin{split}
\frac{1}{4} \frac{\mathrm{d}}{\mathrm{d}t}\left\|(u^3)^2 \right\|_{2}^{2}+\frac{3}{4}\left\|\nabla (u^3)^2\right\|_{2}^{2}\leq& C\left\|\partial_{3}\bm{u}\right\|_{q}^{p}\left(\left\|\omega^3\right\|_{2}^2+\left\|(u^3)^2\right\|_{2}^2+\left\|\partial_{3}\bm{u}\right\|_{2}^{2}\right)\\
&+\frac{1}{8}\left(\left\|\nabla\omega^3\right\|_{2}^2+\left\|\nabla(u^3)^2\right\|_{2}^2+\left\|\nabla\partial_{3}\bm{u}\right\|_{2}^{2}\right).
\end{split}
\end{equation}

\subsection{Estimate of $\partial_{3}u$}
Taking inner product of the equation \eqref{NS} with $-\partial_{3}^{2}\bm{u}$, we obtain
\begin{equation}\label{3}
\begin{split}
\frac{1}{2}\frac{\mathrm{d}}{\mathrm{d}t}\left\|\partial_{3}\bm{u}\right\|_{2}^{2}+\left\|\nabla \partial_{3}\bm{u}\right\|_{2}^{2}=&-\int \partial_{3}\bm{u}\cdot\nabla \bm{u}\cdot\partial_{3}\bm{u}~\mathrm{d}x \\
=&-\underset{h=1,2 }{\sum}\ \int \partial_{3}\bm{u}\cdot\nabla u^{h}\cdot\partial_{3}u^{h}~\mathrm{d}x\\
&-\int\partial_{3}\bm{u}\cdot\nabla u^{3}\cdot\partial_{3}u^{3}~\mathrm{d}x\\
=&:K_{1}+K_{2}.
\end{split}
\end{equation}
By H\"{o}lder inequality and Lemma \ref{lemma2.1}, we have
\begin{equation}\label{K_{1}}
\begin{split}
K_{1}&\leq C  \left \|\partial_{3}\bm{u}\right\|_{q}\left\|\nabla u^{h}\right\|_{\frac{2q}{q-1}}\left\|\partial_{3}\bm{u}\right\|_{\frac{2q}{q-1}} \\
&\leq C \left\|\partial_{3}\bm{u}\right\|_{q}\left(\left\|\omega^3\right\|_{\frac{2q}{q-1}}+\left\|\partial_{3}\bm{u}\right\|_{\frac{2q}{q-1}}\right)
\left\|\partial_{3}\bm{u}\right\|_{\frac{2q}{q-1}} \\
&\leq C\left\|\partial_{3}\bm{u}\right\|_{q}^{p}\left(\left\|\omega^3\right\|_{2}^2+\left\|\partial_{3}\bm{u}\right\|_{2}^{2}\right)+\frac{1}{16}\left(\left\|\nabla\omega^3\right\|_{2}^2+\left\|\nabla\partial_{3}\bm{u}\right\|_{2}^{2}\right),
\end{split}
\end{equation}
which is similar to \eqref{I_{3}}. And by H\"{o}lder, interpolation and Cauchy--Schwarz inequalities, we obtain
\begin{equation}\label{K_{2}}
\begin{split}
K_{2}&=\int~\partial_{3}\bm{u}\cdot u^{3} \cdot\nabla\partial_{3}u^{3}\mathrm{d}x \\
&\leq \left\|\partial_{3}\bm{u}\right\|_{\frac{12q}{q+6}}\left\|u^{3}\right\|_{\frac{12q}{5q-6}}\left\|\nabla\partial_{3}u^{3}\right\|_{2} \\
&\leq C\left\|\partial_{3}\bm{u}\right\|_{q}^{p}\left\|(u^{3})^2\right\|_{2}^{2}+\frac{1}{16}\left(\left\|\nabla(u^{3})^2\right\|_{2}^{2}+\left\|\nabla\partial_{3}\bm{u}\right\|_{2}^{2}\right),
\end{split}
\end{equation}
which is similar to \eqref{I_{1,2}}.

Summing up \eqref{3}, \eqref{K_{1}}, and \eqref{K_{2}}, we get
\begin{equation}\label{K}
\begin{split}
\frac{1}{2}\frac{\mathrm{d}}{\mathrm{d}t}\left\|\partial_{3}\bm{u}\right\|_{2}^{2}+\left\|\nabla \partial_{3}\bm{u}\right\|_{2}^{2}\leq& C \left\|\partial_{3} \bm{u}\right\|_{q}^{p}\left(\left\|\omega^{3}\right\|_{2}^{2}+\left\|(u^{3})^2\right\|_{2}^{2}+\left\|\partial_{3}\bm{u}\right\|_{2}^{2}\right)\\
&+\frac{1}{8}\left(\left\|\nabla \omega^{3}\right\|_{2}^{2}+\left\|\nabla(u^{3})^2\right\|_{2}^{2}+\left\|\nabla\partial_{3}\bm{u}\right\|_{2}^{2}\right).
\end{split}
\end{equation}

\section{Some \textit{a priori} estimates for $2\leq q \leq 6$}

\subsection{Estimate of $\omega^3$ and $\partial_{3} u$}

Recall that
\begin{equation} \label{1'}
\frac{1}{2}\frac{\mathrm{d}}{\mathrm{d}t}\left(\left\|\omega^3\right\|_{2}^{2}+\left\|\partial_{3}\bm{u}\right\|_{2}^{2}\right)+\left\|\nabla \omega^3\right\|_{2}^{2}+\left\|\nabla\partial_{3}\bm{u}\right\|_{2}^{2}=:\ I_{1}+I_{2}+I_{3}+K_{1}+K_{2},
\end{equation}
in \eqref{1} and \eqref{3}.

$\bullet$ For $2\leq q <3$.

\begin{equation}\label{I'_{1,2}1}
\begin{split}
I_{1}+I_{2}+K_{2}\leq& 2 \left\|\partial_{3}\bm{u}\right\|_{q}^{\frac{2}{3}}\cdot\left\|\partial_{3}\bm{u}\right\|_{6}^{\frac{1}{3}} \left\|\left( u^{3}\right)^{\frac{3}{2}}\right\|_{\frac{3q}{2q-3}}^{\frac{2}{3}} \left(\left\|\nabla \omega^{3}\right\|_{2}+\left\|\nabla\partial_{3}\bm{u}\right\|_{2}\right) \\
\leq& C \left\|\partial_{3}\bm{u}\right\|_{q}^{\frac{2}{3}}\left\|\nabla\partial_{3}\bm{u}\right\|_{2}^{\frac{1}{3}} \left\|\left( u^{3}\right)^{\frac{3}{2}}\right\|_{3}^{\frac{5q-9}{3q}}\left\|\left( u^{3}\right)^{\frac{3}{2}}\right\|_{9}^{\frac{3-q}{q}}\\
&\times\left(\left\|\nabla \omega^{3}\right\|_{2}+\left\|\nabla\partial_{3}\bm{u}\right\|_{2}\right).
\end{split}
\end{equation}

Summing up \eqref{I_{3}}, \eqref{K_{1}}, \eqref{1'} and \eqref{I'_{1,2}1}, we have
\begin{equation} \label{I'1}
\begin{split}
&\left\|\omega^3\right\|_{L^{\infty,2}_{T_{1},T^*}}^{2}+\left\|\partial_{3}\bm{u}\right\|_{L^{\infty,2}_{T_{1},T^*}}^{2}+2\left\|\nabla \omega^3\right\|_{L^{2,2}_{T_{1},T^*}}^{2}+2\left\|\nabla \partial_{3}\bm{u}\right\|_{L^{2,2}_{T_{1},T^*}}^{2}\\
\leq & 2\left\|\omega^3(T_{1})\right\|_{2}^{2}+2\left\|\partial_{3}\bm{u}(T_{1})\right\|_{2}^{2}+C \left\|\partial_{3}\bm{u}\right\|_{L^{p,q}_{T_{1},T^*}}^{\frac{2}{3}}\left\|\nabla\partial_{3}\bm{u}\right\|_{L^{2,2}_{T_{1},T^*}}^{\frac{1}{3}} \left\|\left( u^{3}\right)^{\frac{3}{2}}\right\|_{L^{\infty,3}_{T_{1},T^*}}^{\frac{5q-9}{3q}}\\
&\times\left\|\left( u^{3}\right)^{\frac{3}{2}}\right\|_{L^{3,9}_{T_{1},T^*}}^{\frac{3-q}{q}}\left(\left\|\nabla \omega^{3}\right\|_{L^{2,2}_{T_{1},T^*}}+\left\|\nabla \partial_{3}\bm{u}\right\|_{L^{2,2}_{T_{1},T^*}}\right)+C \left\|\partial_{3} \bm{u}\right\|_{L^{p,q}_{T_{1},T^*}}^{p}\\
&\times\left(\left\|\omega^{3}\right\|_{L^{\infty,2}_{T_{1},T^*}}^{2}+\left\|\partial_{3}\bm{u}\right\|_{L^{\infty,2}_{T_{1},T^*}}^{2}\right)+\frac{1}{2}\left(\left\|\nabla \omega^{3}\right\|_{L^{2,2}_{T_{1},T^*}}^2+\left\|\nabla \partial_{3}\bm{u}\right\|_{L^{2,2}_{T_{1},T^*}}^2\right) \\
\leq& 2\left\|\omega^3(T_{1})\right\|_{2}^{2}+2\left\|\partial_{3}\bm{u}(T_{1})\right\|_{2}^{2}+ C \left(\left\|\partial_{3} \bm{u}\right\|_{L^{p,q}_{T_{1},T^*}}^{p}+\left\|\partial_{3} \bm{u}\right\|_{L^{p,q}_{T_{1},T^*}}^{\frac{4q}{5q-9}}\right)\\
&\cdot\left(\left\|\omega^{3}\right\|_{L^{\infty,2}_{T_{1},T^*}}^{2}+\left\|\left( u^{3}\right)^{\frac{3}{2}}\right\|_{L^{\infty,3}_{T_{1},T^*}}^2+\left\|\partial_{3}\bm{u}\right\|_{L^{\infty,2}_{T_{1},T^*}}^{2}\right)+\left\|\nabla\omega^{3}\right\|_{L^{2,2}_{T_{1},T^*}}^{2}\\
&+\left\|\left( u^{3}\right)^{\frac{3}{2}}\right\|_{L^{3,9}_{T_{1},T^*}}^2+\left\|\nabla\partial_{3}\bm{u}\right\|_{L^{2,2}_{T_{1},T^*}}^2 .
\end{split}
\end{equation}

$\bullet$ For $3\leq q \leq 6$.

According to H\"{o}lder, interpolation and Cauchy--Schwarz inequalities, we get
\begin{equation}\label{I'_{1,2}2}
\begin{split}
I_{1}+I_{2}+K_{2}&\leq 2 \left\|\partial_{3}\bm{u}\right\|_{q}^{\frac{2}{3}}\cdot\left\|\partial_{3}\bm{u}\right\|_{\frac{6q}{5q-12}}^{\frac{1}{3}} \left\|\left( u^{3}\right)^{\frac{3}{2}}\right\|_{3}^{\frac{2}{3}}\left(\left\|\nabla \omega^{3}\right\|_{2}+\left\|\nabla\partial_{3}\bm{u}\right\|_{2}\right) \\
&\leq C \left\|\partial_{3}\bm{u}\right\|_{q}^{\frac{2}{3}}\cdot\left\|\partial_{3}\bm{u}\right\|_{2}^{\frac{2}{3}-\frac{2}{q}}\left\|\nabla\partial_{3}\bm{u}\right\|_{2}^{\frac{2}{q}-\frac{1}{3}} \left\|\left( u^{3}\right)^{\frac{3}{2}}\right\|_{3}^{\frac{2}{3}}\left(\left\|\nabla \omega^{3}\right\|_{2}+\left\|\nabla\partial_{3}\bm{u}\right\|_{2}\right) \\
&\leq C\left\|\partial_{3}\bm{u}\right\|_{q}^{p}\left(\left\|(u^{3})^{\frac{3}{2}}\right\|_{3}^{2}+\left\|\partial_{3}\bm{u}\right\|_{2}^2\right)+\frac{1}{16}\left(\left\|\nabla \omega^{3}\right\|_{2}^{2}+\left\|\nabla\partial_{3}\bm{u}\right\|_{2}^{2}\right).
\end{split}
\end{equation}
Summing up \eqref{I_{3}}, \eqref{K_{1}}, \eqref{1'} and \eqref{I'_{1,2}2},  we obtain
\begin{align}\label{I'2}
\begin{split}
&\frac{1}{2}\frac{\mathrm{d}}{\mathrm{d}t}\left(\left\|\omega^3\right\|_{2}^{2}+\left\|\partial_{3}\bm{u}\right\|_{2}^{2}\right)+\left\|\nabla \omega^3\right\|_{2}^{2}+\left\|\nabla \partial_{3}\bm{u}\right\|_{2}^{2}\\
&\leq C \left\|\partial_{3} \bm{u}\right\|_{q}^{p}\left(\left\|\omega^{3}\right\|_{2}^{2}+\left\|(u^{3})^{\frac{3}{2}}\right\|_{3}^{2}+\left\|\partial_{3}\bm{u}\right\|_{2}^2\right)\\
&+\frac{1}{8}\left(\left\|\nabla \omega^{3}\right\|_{2}^{2}+\left\|\nabla\partial_{3}\bm{u}\right\|_{2}^{2}\right).
\end{split}
\end{align}

\subsection{Estimate of $u^3$}

Taking inner product of the equation \eqref{u^3} with $|u^3|^{\frac{5}{2}}u^3$, we obtain
\begin{equation}\label{2'}
\begin{split}
\frac{2}{9}\frac{\mathrm{d}}{\mathrm{d}t}\left\||u^3|^{\frac{9}{4}}\right\|_{2}^{2}+\frac{56}{81}\left\|\nabla |u^3|^{\frac{9}{4}}\right\|_{2}^{2}=&-\int\partial_{3}\Pi\cdot |u^3|^{\frac{5}{2}}u^3~\mathrm{d}x\\
=&2 \ \underset{i,j=1,2,3}{\sum}\ \int \Delta^{-1}\partial_{i}\partial_{j}\left(\partial_{3}u^{i}u^{j}\right)|u^3|^{\frac{5}{2}}u^3~\mathrm{d}x \\
=&2\ \underset{i=1,2,3 }{\sum} \ \int \Delta^{-1}\partial_{i}\partial_{3}\left(\partial_{3}u^{i}u^{3}\right)|u^3|^{\frac{5}{2}}u^3~\mathrm{d}x \\
&+2\ \underset{\substack{i=1,2,3\\h=1,2} }{\sum}\ \int \Delta^{-1}\partial_{i}\partial_{h}\left(\partial_{3}u^{i}u^{h}\right)|u^3|^{\frac{5}{2}}u^3~\mathrm{d}x \\
=&:J_{1}+J_{2}.
\end{split}
\end{equation}
Applying H\"{o}lder, interpolation and Cauchy--Schwarz inequalities, we have
\begin{equation}\label{J'_{1}}
\begin{split}
J_{1}\leq&C \left\|\partial_{3}\bm{u}\cdot u^3\right\|_{\frac{9q}{2q+7}}\left\||u^3|^{\frac{5}{2}}u^3\right\|_{\frac{9q}{7q-7}}\\
\leq&C \left\|\partial_{3}\bm{u}\right\|_{q}\left\||u^3|^{\frac{9}{4}}\right\|_{\frac{2q}{q-1}}^{2}\\
\leq& C \left\|\partial_{3} \bm{u}\right\|_{q}^{p}\left\||u^3|^{\frac{9}{4}}\right\|_{2}^{2}+\frac{1}{16}\left\|\nabla|u^3|^{\frac{9}{4}}\right\|_{2}^{2}.
\end{split}
\end{equation}

By Lemma \ref{lemma2.4}, we obtain
\begin{equation}\label{J'_{2}}
\begin{split}
J_{2}\leq& C \left\|\partial_{3}\bm{u}\right\|_{q}\left\|u^h\right\|_{3q}\left\||u^3|^{\frac{9}{4}}\right\|_{\frac{14q}{9q-12}}^{\frac{14}{9}}\\
\leq& C \left\|\partial_{3}\bm{u}\right\|_{q}^{\frac{4}{3}}\left\|\nabla u^h\right\|_{2}^{\frac{2}{3}}\left\||u^3|^{\frac{9}{4}}\right\|_{2}^{\frac{4(5q-9)}{9q}}\left\|\nabla|u^3|^{\frac{9}{4}}\right\|_{2}^{\frac{2(6-q)}{3q}}\\
\leq&C\left\|\partial_{3}\bm{u}\right\|_{q}^{p}\left\|\nabla u^h\right\|_{2}^{\frac{q}{2q-3}}\left\||u^3|^{\frac{9}{4}}\right\|_{2}^{\frac{2(5q-9)}{6q-9}}+\frac{1}{16}\left\|\nabla|u^3|^{\frac{9}{4}}\right\|_{2}^{2}.
\end{split}
\end{equation}

Summing up \eqref{2'}, \eqref{J'_{1}} and \eqref{J'_{2}}, we get
\begin{equation}
\frac{\mathrm{d}}{\mathrm{d}t}\left\||u^3|^{\frac{9}{4}}\right\|_{2}^{2}+\left\|\nabla |u^3|^{\frac{9}{4}}\right\|_{2}^{2}\leq C\left\|\partial_{3}\bm{u}\right\|_{q}^{p} \left(\left\|\nabla u^h\right\|_{2}^{\frac{q}{2q-3}}\left\||u^3|^{\frac{9}{4}}\right\|_{2}^{\frac{2(5q-9)}{6q-9}}+\left\||u^3|^{\frac{9}{4}}\right\|_{2}^{2}\right).
\end{equation}

$\bullet$ For $2\leq q <3$.

By Lemma \ref{lemma2.1}, we deduce that

\begin{equation}\label{J'1}
\begin{split}
&\left\||u^3|^{\frac{3}{2}}\right\|_{L^{\infty,3}_{T_{1},T^*}}^{2}+
\left\||u^3|^{\frac{3}{2}}\right\|_{L^{3,9}_{T_{1},T^*}}^{2}\\
\leq&2\left\||u^3(T_{1})|^{\frac{3}{2}}\right\|_{3}^{2}+ C\left\|\partial_{3}\bm{u}\right\|_{L^{p,q}_{T_{1},T^*}}^{\frac{2p}{3}} \left(\left\|\nabla u^h\right\|_{L^{\infty,2}_{T_{1},T^*}}^{\frac{2q}{6q-9}}\left\||u^3|^{\frac{3}{2}}\right\|_{L^{\infty,3}_{T_{1},T^*}}^{\frac{10q-18}{6q-9}}+\left\||u^3|^{\frac{3}{2}}\right\|_{L^{\infty,3}_{T_{1},T^*}}^{2}\right)\\
\leq&2\left\||u^3(T_{1})|^{\frac{3}{2}}\right\|_{3}^{2}+ C\left\|\partial_{3}\bm{u}\right\|_{L^{p,q}_{T_{1},T^*}}^{\frac{2p}{3}} \left(\left\|\nabla u^h\right\|_{L^{\infty,2}_{T_{1},T^*}}^{2}+\left\||u^3|^{\frac{3}{2}}\right\|_{L^{\infty,3}_{T_{1},T^*}}^{2}\right) \\
\leq & 2\left\||u^3(T_{1})|^{\frac{3}{2}}\right\|_{3}^{2}+ C\left\|\partial_{3}\bm{u}\right\|_{L^{p,q}_{T_{1},T^*}}^{\frac{2p}{3}} \left(\left\|\partial_{3}\bm{u}\right\|_{L^{\infty,2}_{T_{1},T^*}}^{2}+\left\|\omega^3\right\|_{L^{\infty,2}_{T_{1},T^*}}^{2}+\left\||u^3|^{\frac{3}{2}}\right\|_{L^{\infty,3}_{T_{1},T^*}}^{2}\right).
\end{split}
\end{equation}

$\bullet$ For $3\leq q \leq 6$.

By Lemma \ref{lemma2.1}, we deduce that

\begin{equation}\label{J'2}
\begin{split}
\frac{\mathrm{d}}{\mathrm{d}t}\left\||u^3|^{\frac{3}{2}}\right\|_{3}^{2}\leq& C\left\|\partial_{3}\bm{u}\right\|_{q}^{p} \left(\left\|\nabla u^h\right\|_{2}^{\frac{q}{2q-3}}\left\||u^3|^{\frac{3}{2}}\right\|_{3}^{\frac{3q-6}{2q-3}}+\left\||u^3|^{\frac{3}{2}}\right\|_{3}^{2}\right)\\
\leq& C\left\|\partial_{3}\bm{u}\right\|_{q}^{p} \left(\left\|\nabla u^h\right\|_{2}^{2}+\left\||u^3|^{\frac{3}{2}}\right\|_{3}^{2}\right) \\
\leq &  C\left\|\partial_{3}\bm{u}\right\|_{q}^{p} \left(\left\|\partial_{3}\bm{u}\right\|_{2}^{2}+\left\|\omega^3\right\|_{2}^{2}+\left\||u^3|^{\frac{3}{2}}\right\|_{3}^{2}\right).
\end{split}
\end{equation}

\section{Proof of Theorem \ref{thm}}
Assume that $\mathbf{u}\in C([0,T^*);H^1(\R^3))\cap L^2_{loc}([0,T^*);H^2(\R^3))$ be the unique solution of the Navier--Stokes equations, while $T^* \leq T$ is the maximal point.

There exists $T_{1}<T^*$, such that $\left\|\partial_{3} \bm{u}\right\|_{L^{p,q}_{T_{1},T^*}}\leq \epsilon<1$, while the constant $\epsilon$ is sufficient small.

Denote
\begin{align}
E(t)&=\left\|\omega^{3}\right\|_{2}^{2}+\left\|\partial_{3}\bm{u}\right\|_{2}^{2},\\
E_{1}(t)&=E(t)+\left\|(u^{3})^2\right\|_{2}^{2},\\
E_{2}(t)&=E(t)+\left\|(u^{3})^\frac{3}{2}\right\|_{3}^{2}.
\end{align}
\begin{proof}
	Now, we begin our proof.
	
	$\bullet$  $\partial_{3}\bm{u}$ satisfies \eqref{0} with $\frac{3}{2}<q<2$.
	
	Summing up \eqref{I}, \eqref{J}, \eqref{K} and then using Gronwall inequality, we have
	\begin{equation}\label{4}
	E(t)\leq E_{1}(t)\leq E_{1}(0)\ \exp^{C \int_{0}^{T}\left\|\partial_{3} \bm{u}\right\|_{q}^{p}~\mathrm{d}\tau}<+\infty,
	\end{equation}
	for $0\leq t< T^*$.
	
	$\bullet$  $\partial_{3}\bm{u}$ satisfies \eqref{0} with $2 \leq q < 3$.
	
	For the convenience of the readers, we give another proof for $2 \leq q < 3$, which is similar to \cite{Kukavica2007}.
	
	Summing up \eqref{I'1}, \eqref{J'1}, we obtain
	\begin{equation}
	\underset{T_{1}\leq t<T^*}{\sup} E_{2}(t)\leq 2E_{2}(T_{1})+C_{1}\epsilon \cdot\underset{T_{1}\leq t<T^*}{\sup} E_{2}(t).
	\end{equation}
	Pick $\epsilon$ sufficient small such that $C_{1}\epsilon<\frac{1}{2}$. Then we have
	
	\begin{equation}
	E(t)\leq E_{2}(t)\leq 4E_{2}(T_{1})<+\infty,
	\end{equation}
	for $T_{1}\leq t< T^*$.
	
	$\bullet$  $\partial_{3}\bm{u}$ satisfies \eqref{0} with $3 \leq q \leq 6$.
	
	Summing up \eqref{I'2}, \eqref{J'2} and then using Gronwall inequality, we get
	\begin{equation}\label{6}
	E(t)\leq E_{2}(t)\leq E_{2}(0)\ \exp^{C \int_{0}^{T}\left\|\partial_{3} \bm{u}\right\|_{q}^{p}~\mathrm{d}\tau}<+\infty,
	\end{equation}
	for $0\leq t< T^*$.
	
	Therefore, for $\frac{3}{2}<q \leq 6$, we obtain that $E(t)<\infty$,  $T_{1}\leq t< T^*$.
	
	Now we show the control of the terms $\|\nabla \bm{u}\|_{2}$. Applying the spatial derivative $\nabla$ to the Navier–Stokes equations \eqref{NS}, and then taking $L^2$ inner
	product of the resulting equations with $\nabla \bm{u}$, we obtain
	\begin{align*}
	&\frac{1}{2}\frac{\mathrm{d}}{\mathrm{d}t}\left\|\nabla\bm{u}\right\|_{2}^{2}+\left\|\nabla^2\bm{u}\right\|_{2}^{2}\\
	=&-\underset{\substack{i,j,k=1,2,3} }{\sum}\ \int\partial_{i}u^j \partial_{j}u^k \partial_{i}u^k~\mathrm{d}x\\
	=& -\underset{\substack{i,k=1,2,3} }{\sum}\ \int\partial_{i}u^3 \partial_{3}u^k \partial_{i}u^k~\mathrm{d}x
	-\underset{\substack{h=1,2\\i,k=1,2,3} }{\sum}\ \int\partial_{i}u^h \partial_{h}u^k \partial_{i}u^k~\mathrm{d}x\\
	\leq& C\left\|\partial_{3}\bm{u}\right\|_{2}\left\|\nabla \bm{u}\right\|_{3}\left\|\nabla \bm{u}\right\|_{6}+C\left\|\nabla u^h\right\|_{2}\left\|\nabla \bm{u}\right\|_{3}\left\|\nabla \bm{u}\right\|_{6}\\
	\leq&C \left(\left\|\partial_{3}\bm{u}\right\|_{2}+\left\|\omega^3\right\|_{2}\right)\left\|\nabla \bm{u}\right\|_{2}^{\frac{1}{2}}\left\|\nabla^2 \bm{u}\right\|_{2}^{\frac{3}{2}}\\
	\leq&C \left(\left\|\partial_{3}\bm{u}\right\|_{2}+\left\|\omega^3\right\|_{2}\right)^4\left\|\nabla \bm{u}\right\|_{2}^{2}+\frac{1}{2}\left\|\nabla^2 \bm{u}\right\|_{2}^{2}.
	\end{align*}
	Applying Gronwall inequality, we obtain that, for $T_{1}\leq t<T^*$,
	\begin{align*}
	\left\|\nabla\bm{u}(t)\right\|_{2}^{2}\leq \left\|\nabla\bm{u}(T_{1})\right\|_{2}^{2}\exp^{C\int_{T_{1}}^{T^*} E(\tau)^4~\mathrm{d}\tau}<+\infty.
	\end{align*}
	We obtain that $\mathbf{u}$ can be continued beyond $T^*$, which contradicts with the definition of $T^*$. Thus, $T^*>T$,
	which yields the results.
	
	$\bullet$ We assume the solution $\bm{u}$ is axisymmetric and $\partial_{3} u^3$ satisfies \eqref{0.1}. By 1D Hardy inequality (See e.g. Lemma 2.1 in \cite{Chen2017}), we deduce that
	$$
	\int_{0}^{\infty}(\frac{u^{r}}{r})^{q}~r \mathrm{d} r<C(q)^{q}\int_{0}^{\infty}(\frac{\partial_{r}(ru^r)}{r})^{q}~r \mathrm{d}r,
	$$
	and
	$$
	\left\|\frac{u^{r}}{r}\right\|_{L^{q}(\R^3)}<C(q)\left\|\partial_{3}u^3\right\|_{L^{q}(\R^3)},
	$$
	since $\partial_{3}u^3=-\frac{\partial_{r}(ru^r)}{r}$. Then $\mathbf{u}$ is regular, by Theorem 1.1 in \cite{Kubica2012}.
\end{proof}
\section*{Acknowledgments}
Hui Chen was supported by Zhejiang Province Science Fund for Youths [LQ19A010002]. Daoyuan Fang was supported by NSF of China [11671353]. Ting Zhang was in part supported by NSF of China [11771389, 11931010, 11621101], Zhejiang Provincial Natural Science Foundation of China [LR17A010001].
 
\end{document}